\documentclass[11pt,reqno]{amsart}
\usepackage{graphicx}
\usepackage{verbatim}
\usepackage{textcomp}
\usepackage{amssymb}
\usepackage{cite}
\usepackage{amsmath}
\usepackage{latexsym}
\usepackage{amscd}
\usepackage{amsthm}
\usepackage{mathrsfs}
\usepackage{xypic}
\usepackage{bm}
\usepackage{url}

\newtheorem{thm}{Theorem}
\newtheorem{corr}[thm]{Corollary}
\newtheorem{lem}[thm]{Lemma}
\newtheorem{prop}[thm]{Proposition}

\theoremstyle{definition}

\theoremstyle{remark}
\newtheorem*{ack}{Acknowledgment}

\def\R{\mathbb R}

\def\f{\frac}
\def\td{\tilde}
\def\ra{\rightarrow}
\def\pt{\partial}

\begin{document}
\title[$f$-minimal surface and manifold with positive $Ric_f^m$]{$f$-minimal surface and manifold with positive $m$-Bakry-\'{E}mery Ricci curvature}
\author{Haizhong Li}
\address{Department of mathematical sciences, and Mathematical Sciences
Center, Tsinghua University, 100084, Beijing, P. R. China}
\email{hli@math.tsinghua.edu.cn}
\author{Yong Wei}
\address{Department of mathematical sciences, Tsinghua University, 100084, Beijing, P. R. China}
\email{wei-y09@mails.tsinghua.edu.cn}
\thanks{The research of the first author was supported by NSFC No. 11271214.}
\subjclass[2010]{{53C42}, {53C21}}
\keywords{$f$-mean curvature, $f$-minimal, $m$-Bakry-\'{E}mery Ricci curvature, eigenvalue estimate}

\maketitle

\begin{abstract}
In this paper, we first prove a compactness theorem for the space of closed embedded $f$-minimal surfaces of fixed topology in a closed three-manifold with positive Bakry-\'{E}mery Ricci curvature. Then we give a Lichnerowicz type lower bound of the first eigenvalue of the $f$-Laplacian on compact manifold with positive $m$-Bakry-\'{E}mery Ricci curvature, and prove that the lower bound is achieved only if the manifold is isometric to the $n$-shpere, or the $n$-dimensional hemisphere. Finally, for compact manifold with positive $m$-Bakry-\'{E}mery Ricci curvature and $f$-mean convex boundary, we prove an upper bound for the distance function to the boundary, and the upper bound is achieved if only if the manifold is isometric to an Euclidean ball.
\end{abstract}

\section{Introduction}

Let $(N^n,g)$ be a smooth Riemannian manifold  and $f$ be a smooth function on $N$. We denote $\bar{\nabla},\bar{\Delta}$ and $\bar{\nabla}^2$ the gradient, Laplacian and Hessian operator on $N$ with respect to $g$, respectively. In this paper by the Bakry-\'{E}mery Ricci curvature we mean
\begin{align}
    Ric_f=Ric+\bar{\nabla}^2f,
\end{align}
which is also called $\infty$-Bakry-\'{E}mery Ricci curvature, i.e., the $m=\infty$ case of the following $m$-Bakry-\'{E}mery Ricci curvature \cite{BM} defined by
\begin{align}\label{m-bakry}
    Ric_f^m=Ric_f-\f 1{m-n}\bar{\nabla} f\otimes\bar{\nabla} f,\qquad (m> n).
\end{align}
When $m=n$, we let $f$ be constant and define $Ric_f^m=Ric$. The equation $Ric_f=\kappa g$ for some constant $\kappa$ is just the gradient Ricci soliton equation, which palys an important role in the study of Ricci flow (see \cite{Cao}). The equation $Ric_f^m=\kappa g$ corresponds to the quasi-Einstein equation (cf.\cite{CSW}), which has been studied by many authors. Denote $dv$ the Riemannian volume form on $N$ with respect to $g$, then $(N^n,g,e^{-f}dv)$ is often called a smooth metric measure space. We refer the interested readers to \cite{Wei-Wylie} for further motivation and examples of the metric measure spaces.

Let $M$ be a hypersurface in $N$ and $\nu$ the outer unit normal vector to $M$. Define the second fundamental form of $M\subset N$ by $h(X,Y)=\langle \bar{\nabla}_X\nu,Y\rangle$ for any two tangent vector fields $X$ and $Y$ on $M$,  and the mean curvature by $H=tr(h)$.  The $f$-mean curvature (see \cite[page 398]{Wei-Wylie}) at a point $x\in M$ with respect to $\nu$ is given by
 \begin{align}
    H^f(x)=H(x)-\langle\bar{\nabla}f(x),\nu(x)\rangle.
 \end{align}
$M$ is called a $f$-minimal hypersurface in $N$ if its $f$-mean curvature $H^f$ vanishes everywhere.

The most well known example of metric measure space is the Gaussian soliton: $(\R^n,g_0,e^{-\f 14|x|^2}dv)$, where $g_0$ is the standard Euclidean metric on $\R^n$. The Gaussian soliton satisfies $Ric_f=\f 12g_0$. The $f$-minimal hypersurface in the Gaussian soliton is the self-shrinker $M^{n-1}\subset\R^n$ which satisfies:
\begin{align*}
    H=\f 12\langle x,\nu\rangle.
\end{align*}
Self-shrinkers play an important role in the mean curvature flow, as they correspond to the self-similar solution to mean curvature flow, and also describe all possible blow ups at a given singularity.

In \cite{CM09} and \cite{Ding-Xin}, Colding-Minicozzi and Ding-Xin considered the compactness property for the space of self-shrinkers in $\R^3$. In this paper, we prove the following compactness theorem for $f$-minimal surface, which is a generalization of the classical compactness of minimal surfaces in closed three manifold with positive Ricci curvature by Choi and Schoen \cite{Choi-1}.
\begin{thm}\label{main-thm}
Let $(N^3,g,e^{-f}dv)$ be a closed metric measure space with positive Bakry-\'{E}mery Ricci curvature. Then the space of closed embedded $f$-minimal surfaces of fixed topological type in $N$ is compact in the $C^k$ topology for any $k\geq 2$.
\end{thm}

Here \textit{closed} means compact and without boundary. Note that when $f$ is a constant function, we get the classical Choi-Schoen's theorem (see \cite[Theorem 1]{Choi-1}). We remark that our approach in section 3 to prove Theorem \ref{main-thm} also works for positive $m$-Bakry-\'{E}mery Ricci curvature case, with some slightly adjustments of the Bochner formula and Reilly formula, see \eqref{Bochner-2} and \eqref{Reilly-2}. But since $Ric_f\geq Ric_f^m$, we get no extension results of Theorem \ref{main-thm}. Recently Ailana Fraser and Martin Li \cite{F-Li} proved a compactness theorem for the space of embedded minimal surfaces with free boundary in three-manifold with non-negative Ricci curvature and convex boundary. So it's also interesting to get an analogue result of Theorem \ref{main-thm} for $f$-minimal surfaces with free boundary case.

In our proof of Theorem \ref{main-thm}, one of the key ingredients is the observation that a $f$-minimal hypersurface $M$ is a minimal hypersurface in $N$ with the conformal changed metric $\td{g}=e^{-\f 2{n-1}f}g$. This can be easily seen from the first variation formula of the volume. We will use this observation in the next section to get the singular compactness result. However, Theorem \ref{main-thm} cannot directly follow from Choi-Schoen's compactness theorem for minimal surfaces in three-manifold with positive Ricci curvature. In fact, the Ricci curvature of the conformal changed metric $\td{g}$ may not have a sign: Recall that for $n\geq 3$, the scalar curvature $\td{R}$ of the conformal changed metric $\td{g}=e^{-\f 2{n-1}f}g$ is given by (cf. \cite{Schoen-Yau})
\begin{align*}
    \td{R}=&e^{\f 2{n-1}f}\left(-\f{n-2}{n-1}|\bar{\nabla}f|^2+2\bar{\Delta}f+R\right),
\end{align*}
where $R$ is the scalar curvature of $(N,g)$. Although the positive Bakry-\'{E}mery Ricci curvature assumption implies that $R+\bar{\Delta}f>0$, we cannot conclude that the scalar curvature and then the Ricci curvature of the conformal metric have a sign. For example, the Gaussian soliton $(\R^n,g_0,e^{-\f 14|x|^2}dv)$ has positive Bakry-\'{E}mery Ricci curvature, while the scalar curvature $\td{R}$ of the conformal changed metric $\td{g}=e^{-\f 1{2(n-1)}|x|^2}g_0$ on $\R^n$ is
\begin{align*}
    \td{R}=&e^{\f 1{2(n-1)}|x|^2}\left(n-\f {n-2}{4(n-1)}|x|^2\right),
\end{align*}
which is positive when $|x|$ is small and becomes negative when $|x|$ is large. Therefore the Ricci curvature of $\td{g}$ does not have a sign.

Our proof follows from the standard argument in Choi-Schoen's paper: We first need a first eigenvalue estimate of the $f$-Laplacian $\Delta_f=\Delta-\nabla f\cdot\nabla$ for $f$-minimal surfaces in manifold with positive Bakry-\'{E}mery Ricci curvature. This is also one of the key ingredients in the proof of Theorem \ref{main-thm}, and was proved by Li Ma and Sheng-Hua Du \cite{MD} recently. In section 3, we show that Ma-Du's result holds under a weaker condition, for example, the orientability assumption is not necessary. Then by considering $f$-minimal surface as a minimal surface in $(N,\td{g})$, and combining with Yang-Yau's inequality \cite{YY}, we get an apriori upper bound on the weighted area of the $f$-minimal surface in terms of the topology. This together with the Gauss equation and Gauss-Bonnet theorem gives an upper bound for the total curvature of $f$-minimal surface. Finally, by the singular compactness proposition and a contradiction argument, we get the smooth compactness theorem \ref{main-thm}.

As a corollary of theorem \ref{main-thm}, we get the following curvature estimates. The proof is by using a contradiction argument like in Choi-Schoen's paper\cite{Choi-1}.
\begin{corr}
Let $(N^3,g,e^{-f}dv)$ be a closed metric measure space with positive Bakry-\'{E}mery Ricci curvature. There exists a constant $C$ depending only on $N$ and an integer $\chi$ such that if $M$ is a closed embedded $f$-minimal surface of Euler characteristic $\chi$ in $N$, then
\begin{align*}
    \max_M\|h\|\leq~C,
\end{align*}
where $\|h\|$ is the norm of the second fundamental form of $M\subset N$.
\end{corr}

Next, in section 4, we will use the Reilly formula to give a Lichnerowicz type lower bound for the first eigenvalue of $f$-Laplacian on compact manifold with posivite $m$-Bakry-\'{E}mery Ricci curvature. The classical Lichnerowicz theorem \cite{Lich} says that for an $n$-dimensional closed Riemannian manifold with Ricci curvature $Ric\geq (n-1)K>0$, then the first eigenvalue of Lapalcian on $N$ satisfies $\lambda_1(\bar{\Delta})\geq nK$. Obata \cite{Oba} then proved that the equality holds only when $N$ is isometric to the $n$-sphere of radius $1/{\sqrt{K}}$. This was generalized by Reilly \cite{Reil} to compact manifold with mean-convex boundary for Dirichlet problem,  and by Escobar \cite{Esco} to compact manifold with convex boundary for Neumann problem. The same lower bound for $\lambda_1(\bar{\Delta})$ holds and the equality holds when $N$ is isometric to the $n$-dimensional hemisphere of radius $1/{\sqrt{K}}$. The following theorem shows that a similar result also holds for manifold with positive $m$-Bakry-\'{E}mery Ricci curvature and with some suitable boundary condition. We remark that the result $\lambda_1\geq mK$ in theorem \ref{lichn} was essentially proved in \cite{MD}, just with some slightly different expressions. Our contribution is the rigidity result when the equality $\lambda_1= mK$ holds.
\begin{thm}\label{lichn}
Let $(N^n,g)$ be an $n$-dimensional Riemannian manifold (possibly with boundary $\pt N$) and $f$ be a smooth function on $N$. Assume that the $m$-Bakry-\'{E}mery Ricci curvature satisfies $Ric_f^m\geq (m-1)K>0$. Furthermore, if the boundary $\pt N$ is nonempty, for Dirichlet problem we assume $f$-mean curvature on $\pt N$ is nonnegative; for Neumann problem we assume the boundary $\pt N$ is weakly convex, i.e., the second fundamental form $h\geq 0$ on $\pt N$. Then the first eigenvalue $\lambda_1$ of the $f$-Laplacian on $N$ satisfies:
\begin{align}\label{lichn-2}
    \lambda_1\geq ~mK.
\end{align}
Moreover, Equality is attained only when $m=n$, $f$ is constant and $Ric_f^m=Ric$. In this case, if $N$ has no boundary, then $N$ is the $n$-sphere of radius $1/{\sqrt{K}}$; if $N$ has nonempty boundary, then $N$ is the $n$-dimensional hemisphere of radius $1/{\sqrt{K}}$.
\end{thm}
One can compare theorem \ref{lichn} with Bakry-Qian's eigenvalue comparison results in \cite{BaQ}, which states that the first eigenvalue $\lambda_1$ (with Neumann boundary condition when the boundary is nonempty) of $f$-Laplacian is bounded from below by the first eigenvalue of a one-dimensional model. See also \cite{BN} and \cite{FLL} for more recent results about the first eigenvalue of $f$-Laplacian on manifolds with positive Bakry-\'{E}mery curvature.

We remark that when $Ric_f^m\geq (m-1)K>0$, $(N^n,g)$ is automatically compact and the diameter of $N$ satisfies $diam(N)\leq \pi/{\sqrt{K}}$, see \cite[Theorem 5]{Qian}. So we don't need to assume $N$ is compact in theorem \ref{lichn}. In \cite{Ruan}, Ruan proved that when $diam(N)$ is equal to $\pi/{\sqrt{K}}$, then $(N,g)$ is isometric to the $n$-sphere of radius $1/{\sqrt{K}}$. In section 5, we will prove a similar result for manifolds with nonnegative $m$-Bakry-\'{E}mery Ricci curvature and $f$-mean convex boundary (i.e., the $f$-mean curvature on $\pt N$ is positive).

\begin{thm}\label{thm-4}
Let $(N^n,g)$ be an $n$-dimensional complete Riemannian manifold with nonempty boundary and $f$ be a smooth function on $N$. Assume that the $m$-Bakry-\'{E}mery Ricci curvature is nonnegative on $N$, and the $f$-mean curvature of the  boundary $\pt N$ satisifes $H^f\geq (m-1)K>0$ for some constant $K>0$. Let $d$ denote the distance function on $N$. Then
\begin{align}\label{dist}
    \sup_{x\in N}d(x,\pt N)\leq &\f 1K.
\end{align}
Moreover, if we assume that $\pt N$ is compact, then $N$ is also compact and equality  holds in \eqref{dist} only when $N$ is isometric to an $n$-dimensional Euclidean ball of radius $1/K$.
\end{thm}

Theorem \ref{thm-4} is an analogue result of Theorem 1.1 in \cite{Martin}, where the manifold with nonnegative Ricci curvature and with mean convex boundary was considered. Our proof of theorem \ref{thm-4} follows the arguments in \cite{Martin}, with some ajustments. As in \cite{Martin}, we also conjecture that the uniform boundary convexity could make $\pt N$ to be compact and hence $N$ would also be compact.

\begin{ack}The first author is grateful to the Department of Mathematics at K. U. Leuven, where part of this work was carried out.
\end{ack}

\section{Reilly formula on metric measure space}

In this section, we first exhibit the Reilly formulas on metric measure space, which are the important tools to prove our main theorems.

Let $(N^n,g,e^{-f}dv)$ be a compact metric measure space with boundary $\pt N$. The $f$-Laplacian $\bar{\Delta}_f=\bar{\Delta}-\bar{\nabla}f\cdot\bar{\nabla}$ on $N$ is self-adjoint with respect to the weighted measure $e^{-f}dv$. A simple calculation gives the following Bochner formula (see \cite{Wei-Wylie,MD,MW}) for any function $u\in C^3(N)$:
\begin{align}\label{Bochner}
    \f 12\bar{\Delta}_f|\bar{\nabla} u|^2=|\bar{\nabla}^2u|^2+Ric_f(\bar{\nabla}u,\bar{\nabla}u)+g(\bar{\nabla}u,\bar{\nabla}\bar{\Delta}_fu).
\end{align}
Using the Bochner formula \eqref{Bochner} and integration by part, Li Ma and Sheng-Hua Du \cite{MD} obtained the following Reilly formula:
\begin{align}
    0=&\int_{N}(Ric_f(\bar{\nabla}u,\bar{\nabla}u)-|\bar{\Delta}_fu|^2+|\bar{\nabla}^2u|^2)e^{-f}dv\label{Reilly}\\
    &\quad +\int_{\pt N}\left((\Delta_fu+H^f\f{\pt u}{\pt \nu})\f{\pt u}{\pt \nu}-\langle\nabla u,\nabla\f{\pt u}{\pt \nu}\rangle+h(\nabla u,\nabla u)\right)e^{-f}d\mu.\nonumber
\end{align}
Here, $Ric_f$ is the Bakry-\'{E}mery Ricci tensor of $N$; $dv$ and $d\mu$ are volume forms on $N$ and $\pt N$ respectively. $\bar{\Delta}_f,\bar{\nabla}$ and $\bar{\nabla}^2$ are the $f$-Laplacian, gradient and Hessian on $N$ respectively; $\Delta_f$ and $\nabla$ are the $f$-Laplacian and gradient operators on $\pt N$; $\nu$ is the unit outward normal of $\pt N$; $H^f$ and $h$ are the $f$-mean curvature and second fundamental form of $\pt N$ in $N$ with respect to $\nu$ respectively.

The Bochner formula \eqref{Bochner} looks similar to the classic Bochner formula. However we have a difficulty that $tr(\bar{\nabla}^2u)\neq \bar{\Delta}_fu$. One way to deal with this is to consider the Bochner formula for $m$-Bakry-\'{E}mery Ricci curvature. When $m>n$, let $z=\f mn$ and by a basic algebraic inequality $(a+b)^2\geq \f{a^2}{z}-\f{b^2}{z-1}$ for $z>1$, we have
\begin{align*}
    |\bar{\nabla}^2u|^2\geq \f 1n(\bar{\Delta} u)^2=&\f 1n(\bar{\Delta}_fu+\bar{\nabla}f\cdot\bar{\nabla}u)^2\\
    \geq&\f 1n\left(\f nm(\bar{\Delta}_fu)^2-\f n{m-n}(\bar{\nabla}f\cdot\bar{\nabla}u)^2\right)\\
    =&\f 1m(\bar{\Delta}_fu)^2-\f 1{m-n}(\bar{\nabla}f\cdot\bar{\nabla}u)^2.
\end{align*}
Substituting this into \eqref{Bochner},\eqref{Reilly} and using the definition \eqref{m-bakry} of $m$-Bakry-\'{E}mery Ricci curvature, we get
\begin{align}\label{Bochner-2}
    \f 12\bar{\Delta}_f|\bar{\nabla} u|^2\geq&\f 1m(\bar{\Delta}_fu)^2+Ric_f^m(\bar{\nabla}u,\bar{\nabla}u)+g(\bar{\nabla}u,\bar{\nabla}\bar{\Delta}_fu).
\end{align}
and
\begin{align}
    0\geq&\int_{N}(Ric_f^m(\bar{\nabla}u,\bar{\nabla}u)-\f {m-1}m|\bar{\Delta}_fu|^2)e^{-f}dv\label{Reilly-2}\\
    &\quad +\int_{\pt N}\left((\Delta_fu+H^f\f{\pt u}{\pt \nu})\f{\pt u}{\pt \nu}-\langle\nabla u,\nabla\f{\pt u}{\pt \nu}\rangle+h(\nabla u,\nabla u)\right)e^{-f}d\mu.\nonumber
\end{align}

Note that the Bochner formula \eqref{Bochner-2} looks very similar with the Bochner formula for the Ricci tensor of an $m$-dimensional manifold. This seems to be Bakry-\'{E}mery's motivation \cite{BM} for the definiton of the $m$-Bakry-\'{E}mery Ricci tensor and for their more general curvature dimension inequalities for diffusion operators. See also \cite{XDLi,XMLi} for the Bochner formula \eqref{Bochner-2}.

\section{The space of $f$-minimal surfaces}

In this section, we assume that $(N^n,g,e^{-f}dv)$ is a closed metric measure space with positive Bakry-\'{E}mery Ricci curvature $Ric_f$. We will prove Theorem \ref{main-thm}. First, we need the following lemma of Frankel type , which was stated in G. Wei and W. Wylie's paper (see Theorem 7.4 in \cite{Wei-Wylie}). Here we give an alternative proof using the Reilly formula \eqref{Reilly}.
\begin{lem}\label{lem-1}
Let $(N^n,g,e^{-f}dv)$ be a closed metric measure space with positive $Ric_f$. Then any two closed embedded $f$-minimal hypersurfaces $\Sigma_1$ and $\Sigma_2$ in $N$ must intersect, i.e.,$\Sigma_1\cap\Sigma_2\neq{\O}$. So that any closed embedded $f$-minimal hypersurface in $N$ is connected.
\end{lem}
\begin{proof}
The proof is motivated by Fraser-Li's paper \cite{F-Li}. Suppose $\Sigma_1$ and $\Sigma_2$ are disjoint. Let $\Omega$ be the domain bounded by $\Sigma_1$ and $\Sigma_2$, then $\Omega$ is a compact manifold with boundary $\Sigma_1\cup\Sigma_2$. Consider the following boundary value problem on $\Omega$:
\begin{align}\label{prb-1}
    \left\{\begin{array}{ll}
             \bar{\Delta}_fu=0, & \textrm{on}~\Omega \\
             u=0, & \textrm{on}~\Sigma_1 \\
             u=1, & \textrm{on}~\Sigma_2
           \end{array}\right.
\end{align}
Let $\hat{u}=u-\varphi$, where $\varphi\in C^{\infty}(\Omega)$ satisfying $\varphi=0$ on $\Sigma_1$ and $\varphi=1$ on $\Sigma_2$. Then the above problem is equivalent to the following
 \begin{align}\label{prb-2}
    \left\{\begin{array}{ll}
             \bar{\Delta}_f\hat{u}= \bar{\Delta}_f\varphi, & \textrm{on}~\Omega \\
             \hat{u}=0, & \textrm{on}~\Sigma_1\cup\Sigma_2
           \end{array}\right.
\end{align}
Since $\bar{\Delta}_f\varphi\in C^{\infty}(\Omega)$, the classical results for elliptic equations with homogeneous boundary value imply that \eqref{prb-2} has a solution $\hat{u}\in C^{\infty}(\Omega)$, and therefore $u=\hat{u}+\varphi\in C^{\infty}(\Omega)$ is a solution to \eqref{prb-1}.
Apply $u$ and $\Omega$ to the Reilly formula \eqref{Reilly}, we obtain
\begin{align}\label{eq-2}
    0\geq&~\int_{\Omega}Ric_f(\bar{\nabla}u,\bar{\nabla}u)e^{-f}dv.
\end{align}
The boundary terms for $\Sigma_1\cup\Sigma_2$ vanishes since $\Sigma_1$ and $\Sigma_2$ are $f$-minimal and $u$ is constant on $\Sigma_1$ and $\Sigma_2$. Since $Ric_f$ is positive, \eqref{eq-2} implies $u$ is constant on $\Omega$, which is a contradiction since $u=0$ on $\Sigma_1$ and $u=1$ on $\Sigma_2$.
\end{proof}

In \cite{Law}, Lawson proved that for a closed embedded minimal hypersurface $M$ in a closed manifold $N$ with positive Ricci curvature, if both $M$ and $N$ are orientable, then $N\setminus M$ consists of two components $\Omega_1$ and $\Omega_2$. The following lemma is a generalization of this result to the $f$-minimal case.
\begin{lem}\label{lem-2}
Let $(N^n,g,e^{-f}dv)$ be a closed metric measure space with positive $Ric_f$, and let $M$ be a closed embedded $f$-minimal hypersurface. If both $M$ and $N$ are orientable, then $N\setminus M$ consists of two components $\Omega_1$ and $\Omega_2$.
\end{lem}
\begin{proof}
First we observe that for a compact connected metric measure space $(\Omega,g,e^{-f}dv)$ with boundary $\pt\Omega$, if $Ric_f$ of $\Omega$ is positive and the $f$-mean curvature of the boundary $\pt\Omega$ is nonnegative, then $\pt\Omega$ is connected. This can be proved by a similar argument as lemma \ref{lem-1}: Suppose $\pt\Omega$ is not connected. Let $\Sigma$ be one of its components. Choose a $f$-harmonic function $u$ (i.e., $\bar{\Delta}_f u=0$ on $\Omega$) which is equal to $0$ on $\Sigma$ and is equal to one on $\pt\Omega\setminus \Sigma$. The existence of $u$ is by the classical results for elliptic equations as in the proof of lemma \ref{lem-1}. Then the Reilly formula \eqref{Reilly} implies that $u$ is a constant, which is a contradiction.

To prove lemma \ref{lem-2}, we follow the argument in \cite{Law}. Let $D=N\setminus M$. For any $p\in M$ we have a neighborhood $U$ and local coordinates $(x^1,\cdots,x^n)$ on $U$ such that $M\cap U$ corresponds to the hyperplane $x_1=0$. Then we get a local coordinates for the boundary points of $D^*=D\cup \pt D$ by first considering $x_1\geq 0$ and then $x_1\leq 0$. Note that $D^*$ has positive $Ric_f$ and the $f$-mean curvature of  the boundary is nonnegative since $M$ is $f$-minimal. If $D^*$ were connected, then the boundary of $D^*$ would be connected by the previous paragraph. However, since $M$ is orientable and connected by lemma \ref{lem-1}, we have that $\pt D$ has two components. If follows that $D$ has two components $D_+$ and $D_-$ and that $D^*$ is the disjoint union of $\bar{D}_+$ and $\bar{D}_-$. This completes the proof.
\end{proof}

We remark that although a $f$-minimal hypersurface $M$ can be characterized as a minimal hypersurface in $(N,\td{g})$, Lemma \ref{lem-2} cannot follows directly from the Lawson's result \cite[Theorem 2]{Law}, since we may not have a sign about the Ricci curvature of the conformal changed metric $\td{g}$.

In the following, we will give a lower bound of the first eigenvalue of the $f$-Laplacian on a $f$-minimal hypersurface in closed metric measure space with positive $Ric_f$. Let $M$ be a $f$-minimal hypersurface in $(N^n,g,e^{-f}dv)$. Denote $d{\mu}$ the volume form on $M$ with respect to the metric induced from $(N,g)$. The $f$-Laplacian $\Delta_f=\Delta-\nabla f\cdot\nabla$ is a self-adjoint operator on $M$ with respect to $e^{-f}d{\mu}$. The first eigenvalue $\lambda_1$ of $\Delta_f$ is the lowest nonzero real number which satisfies
\begin{align*}
    -\Delta_fu=\lambda_1 u
\end{align*}
with Dirichlet or Neumann boundary condition if the boundary of $M$ is not empty. By the variational characterization, when $M$ is closed (or for the Neumann problem when $M$ has boundary )we also have
\begin{align}\label{vari}
    \lambda_1=\inf\limits_{\int_Mue^{-f}d\mu=0}\f{\int_M|\nabla u|^2e^{-f}d{\mu}}{\int_Mu^2e^{-f}d{\mu}}.
\end{align}
For Dirichlet problem when $M$ has boundary, the infimum in \eqref{vari} is taken among all smooth functions which vanish on the boundary $\pt M$. Using the Reilly formula, Li Ma and Sheng-Hua Du (\cite[Theorem 3]{MD}) proved that for a closed embedded $f$-minimal hypersurface $M$ in a closed orientable metric measure space $(N^n,g,e^{-f}dv)$ with $Ric_f\geq \kappa>0$, if $M$ divides $N$ into two components, then the first eigenvalue of $f$-Laplacian on $M$ satisfies $\lambda_1\geq \kappa/2$, which generalized a result of Choi and Wang \cite{Choi-2}. Here, using the universal covering space argument and Lemma \ref{lem-1}, we show that Ma-Du's theorem holds under a weaker assumption.
\begin{thm}\label{thm-3}
Let $M$ be a closed embedded $f$-minimal hypersurface in a closed metric measure space $(N^n,g,e^{-f}dv)$ with $Ric_f\geq \kappa>0$. Then the first eigenvalue $\lambda_1$ of the $f$-Laplacian on $M$ satisfies $\lambda_1\geq \kappa/2$.
\end{thm}
\begin{proof}
Let $\td{N}$ be the universal cover of $N$. Then $\td{N}$ satisfies the same curvature assumption as $N$. Since compact manifold with positive $Ric_f$ has finite fundamental group $\pi_1(N)$ (see eg.\cite{FR,XMLi,Wei-Wylie}), $\td{N}$ is compact and $\pi:\td{N}\ra N$ is a finite covering. Let $\td{M}$ be the lifting of $M$. Since $\td{M}$ is embedded and $\td{N}$ is simply connected, both $\td{N}$ and $\td{M}$ are orientable and then $\td{M}$ divides $\td{N}$ into two components by lemma \ref{lem-2}. By Theorem 3 in \cite{MD}, $\lambda_1(\td{M})\geq \kappa/2$. But the pullback of the first eigenfunction of $M$ into $\td{M}$ is again an eigenfunction of $\td{M}$. Therefore $\lambda_1(M)\geq \lambda_1(\td{M})\geq \kappa/2$.
\end{proof}

By combining Theorem \ref{thm-3} with the classical Yang-Yau's result (see \cite{YY,Schoen-Yau}), we get the following volume estimates.
\begin{corr}\label{cor-5}
Let $M$ be a closed embedded $f$-minimal surface of genus $g$ in a closed metric measure space $(N^3,g,e^{-f}dv)$ with $Ric_f\geq \kappa>0$. Then
\begin{align}\label{eq-1}
    \int_Md\td{\mu}\leq&\f{16\pi}{\kappa}(1+g),
\end{align}
where $d\td{\mu}=e^{-f}d\mu$ is the volume form on $M$ with respect to the induced metric from $(N,\td{g})$.
\end{corr}
\begin{proof}
Since $N^3$ is closed and $f\in C^{\infty}(N)$, by possibly adding a constant, we may assume that $f$ is nonnegative. Let $\td{g}=e^{-f}g$, and denote $\td{\nabla},\td{\Delta}$ the gradient and Laplacian on $M$ with respect to the induced metric from $(N^3,\td{g})$. Then the  first eigenvalue $\td{\lambda}_1$ of the Laplacian $\td{\Delta}$ satisfies
\begin{align*}
    \td{\lambda}_1=&\inf\limits_{\int_Mud\td{\mu}=0}\f{\int_M|\td{\nabla}u|^2d\td{\mu}}{\int_Mu^2d\td{\mu}}\\
    \geq&e^{\min_Nf}\inf\limits_{\int_Mue^{-f}d{\mu}=0}\f{\int_M|\nabla u|^2e^{-f}d{\mu}}{\int_Mu^2e^{-f}d{\mu}}\\
    =&\lambda_1e^{\min_Nf}\geq \kappa/2.
\end{align*}
Here we used the relation $|\td{\nabla}u|^2=e^{f}|\nabla u|^2$ in the first inequality; the second equality is due to the variational characterization \eqref{vari} for the first eigenvalue of the $f$-Laplacian; the second inequality is due to Theorem \ref{thm-3} and that $f$ is non-negative. Then from the classical Yang-Yau's inequality
\begin{align*}
    \td{\lambda}_1\int_Md\td{\mu}\leq 8\pi(1+g),
\end{align*}
we get the inequality \eqref{eq-1}.
\end{proof}

From the Gauss equation and the minimality of $M^2$ in $(N^3,\td{g})$, we have
\begin{align}\label{gauss}
    \f 12\|h\|^2=K^N-K^M,
\end{align}
where $\|h\|^2$ is the squared norm of the second fundamental form of $M$ in $(N^3,\td{g})$. $K^N$ and $K^M$ are sectional curvature of $N$ and Gauss curvature of $M$, with respect to $\td{g}$ and the induced metric from $\td{g}$ respectively. Integrating \eqref{gauss} over $M$ with respect to $d\td{\mu}$ and applying the Gauss-Bonnet theorem, we get
\begin{align}
    \int_M\|h\|^2d\td{\mu}=&2\int_MK^Nd\td{\mu}-2\int_MK^Md\td{\mu}\leq C\int_Md\td{\mu}-4\pi\chi(M)\nonumber\\
                     \leq&C\f{16\pi}{\kappa}(1+g)+8\pi(g-1)\label{bd-h}
\end{align}

We will use the next proposition, which shows us how to use the uniform bounds \eqref{eq-1} and \eqref{bd-h} to obtain a singular compactness result (see \cite[Proposition 7.14]{CM2011}, and \cite{Choi-1,CM09})
\begin{prop}\label{singu}
Let $N^3$ be a closed Riemannian three-manifold and $M_i\subset N$ a sequence of closed embedded minimal surfaces of genus $g$ with
\begin{align*}
    \textrm{Aera}(M_i)\leq C_1,\qquad \textrm{and}\quad \int_{M_i}\|h^{M_i}\|^2d\mu_i\leq C_2,
\end{align*}
where $\textrm{Aera}(M_i)$ is the volume of $M_i$ with respect to the induced metric from $N$, and $h^{M_i}$ is the second fundamental form of $M_i\subset N$, $C_1,C_2$ are two constants independent of $i$. Then there exists a finite set of points $\mathcal{S}\subset N$ and a subsequence $M_{i'}$ that converges uniformly in $C^k$ (any $k\geq 2$) topology on compact subsets of $N\setminus \mathcal{S}$ to a smooth embedded minimal surface (possibly with multiplicity) $M\subset N$.
\end{prop}

Now we are in a position to complete the proof of Theorem \ref{main-thm}. Let $M_i\subset N$ be any sequence of closed embedded $f$-minimal surfaces of fixed genus, which is also a sequence of closed embedded minimal surface of fixed genus in $(N^n,\td{g})$. From \eqref{eq-1} and \eqref{bd-h}, we have uniform area and total curvature bounds of $M_i\subset (N^n,\td{g})$ in terms of the genus. Then Proposition \ref{singu} imply that $M_i$ have a subsequence $M_{i'}$ which converges away from finitely many points to a smooth embedded minimal surface $M$ in $(N^n,\td{g})$. Note that $M$ is $f$-minimal in $(N^n,g)$. It remains to show that the convergence holds across these points, i.e, the convergence is smooth everywhere. By Allard's regularity theorem \cite{A1}, this follows from showing that the convergence is of multiplicity one. Note that since $Ric_f>0$ on $N$, $N$ has finite fundamental group $\pi_1(M)$, after passing a finite cover, we may assume $N$ is simply connected. From the proof of Corollary \ref{cor-5}, the first eigenvalue $\td{\lambda}_1$ of $\td{\Delta}$ of the subsequence $M_{i'}$ has a positive lower bound $\kappa/2$. If the convergence is not multiplicity one, for large $i'$, we can construct a test function to show that $\td{\lambda}_1(M_{i'})$ tends to zero, which violates the lower bound of $\td{\lambda}_1$. The detail argument is just the same as Choi-Schoen's paper \cite{Choi-1} (see also \cite{CM2011}), which we omit here.

\section{First eigenvalue of $f$-Laplacian on manifold with positive $m$-Bakry-\'{E}mery Ricci curvature}

In this section, we will use the Reilly formula \eqref{Reilly-2} to prove Theorem  \ref{lichn}.  Let $\bar{\Delta}_fu=-\lambda_1u$, i.e., $u$ is the first eigenfunction of $f$-Laplacian. When $m>n$, from \eqref{Reilly-2} and the boundary condition of $\pt N$, we have
\begin{align}
    \f {m-1}m\lambda_1^2\int_Nu^2e^{-f}dv\geq&\int_{N}Ric_f^m(\bar{\nabla}u,\bar{\nabla}u)e^{-f}dv\label{eq-eigen}\\
    \geq&(m-1)K\int_N|\bar{\nabla}u|^2e^{-f}dv.\nonumber
\end{align}
Dividing by $\int_Nu^2e^{-f}dv$ and using the fact that $\lambda_1=\int_N|\bar{\nabla}u|^2e^{-f}dv/\int_Nu^2e^{-f}dv$ implies
\begin{align*}
    \lambda_1\geq&~mK.
\end{align*}

When $m=n$, since $f$ is constant and $Ric_f^m=Ric$, the inequality \eqref{lichn-2} is due to the classic results by Lichnerowicz \cite{Lich}, Reilly \cite{Reil} and Escobar \cite{Esco}.

Next, we consider the rigidity when the equality holds in \eqref{lichn-2}. When $m>n$, we show that the inequality \eqref{lichn-2} cannot assume equality. Since if $\lambda_1=mK$, then \eqref{eq-eigen} becomes equality and then the Reilly formula \eqref{Reilly-2} also attains equality. Since the algebraic inequality $(a+b)^2\geq \f{a^2}{z}-\f{b^2}{z-1}$ assumes the equality if and only if $(z-1)a+zb=0$ (for $z>1$). We have that
\begin{align}
0=&\bar{\Delta}_fu+\f m{m-n}\bar{\nabla}f\cdot\bar{\nabla}u=\bar{\Delta}u+\f n{m-n}\bar{\nabla}f\cdot\bar{\nabla}u\label{appen-1}
\end{align}
holds everywhere on $N$. Multiplying \eqref{appen-1} with $u$ and integrating on $N$ with respect to $e^{\f n{m-n}f}dv$ give that
\begin{align*}
    0=&\int_{N}u\left(\bar{\Delta}u+\f n{m-n}\bar{\nabla}f\cdot\bar{\nabla}u\right)e^{\f n{m-n}f}dv\\
    =&-\int_N|\bar{\nabla}u|^2e^{\f n{m-n}f}dv+\int_{\pt N}u\f{\pt u}{\pt\nu}e^{\f n{m-n}f}d\mu\\
    =&-\int_N|\bar{\nabla}u|^2e^{\f n{m-n}f}dv,
\end{align*}
where the third equality is due to the boundary condition of $u$. Therefore we have that $u$ is a constant function on $N$, which is a contradiction since $u$ is the first eigenfunction of $f$-Laplacian and cannot be a constant.

So we conclude that the equality holds in \eqref{lichn-2} only when $m=n$, $f$ is constant and $Ric_f^m=Ric$. Then
by Obata \cite{Oba}, Reilly \cite{Reil} and Escobar \cite{Esco}, we complete the proof of theorem \ref{lichn}.

\section{Manifolds with nonnegative $Ric_f^m$ and $f$-mean convex boundary}

In this section, we modify  the argument in \cite{Martin} to give the proof of Theorem \ref{thm-4}. For any point $x\in N$, since $N$ is complete, there exists a geodesic $\gamma:[0,d]\ra N$ parametrized by arc length with $\gamma(0)=x$, $\gamma(d)\in \pt N$ and $d=d(x,\pt N)$. We need to prove $d\leq 1/K$. Choose  an orthonormal basis $e_1,\cdots e_{n-1}$ for $T_{\gamma(d)}\pt N$ and let $e_i(s)$ be the parallel transport of $e_i$ along $\gamma$. Let $V_i(s)=\varphi(s)e_i(s)$ with $\varphi(0)=0$ and $\varphi(d)=1$.  From the first variation formula, we have that for each $1\leq i\leq n-1$
\begin{align*}
    0=\delta\gamma(V_i)=&\langle \gamma'(d),e_i\rangle-\int_0^d\varphi(s)\langle{\gamma}''(s),e_i(s)\rangle ds=\langle \gamma'(d),e_i\rangle,
\end{align*}
which implies that $\gamma'(d)$ is orthogonal to $\pt N$ at $\gamma(d)$. The second variation formula gives that
\begin{align*}
    \sum_{i=1}^{n-1}\delta^2\gamma(V_i,V_i)=&\int_0^d\left((n-1)\varphi'(s)^2-\varphi(s)^2Ric(\gamma'(s),\gamma'(s))\right)ds-H(\gamma(d))\geq 0
\end{align*}
By the definition of $m$-Bakry-\'{E}mery Ricci curvature, we have
\begin{align*}
    0\leq&\int_0^d\left((n-1)\varphi'(s)^2-\varphi(s)^2Ric_f^m(\gamma'(s),\gamma'(s))\right)ds-H(\gamma(d))\\
    &\qquad+\int_0^d\varphi(s)^2\left(\bar{\nabla}^2f(\gamma'(s),\gamma'(s))-\f 1{m-n}\langle\bar{\nabla}f(\gamma(s)),\gamma'(s)\rangle^2\right)ds
\end{align*}
Since $Ric_f^m$ is nonnegative, by using the facts that
\begin{align*}
    \f{d}{ds}f(\gamma(s))=\langle\bar{\nabla}f(\gamma(s)),\gamma'(s)\rangle
\end{align*}
and that
\begin{align*}
    \f{d^2}{ds^2}f(\gamma(s))=&\bar{\nabla}^2f(\gamma'(s),\gamma'(s)),
\end{align*}
and by integration by parts, we deduce that
\begin{align*}
    0\leq&\int_0^d\biggl((n-1)\varphi'(s)^2-2\varphi(s)\varphi'(s)\langle\bar{\nabla}f(\gamma(s)),\gamma'(s)\rangle\\
    &\qquad -\f 1{m-n}\varphi(s)^2\langle\bar{\nabla}f(\gamma(s)),\gamma'(s)\rangle^2\biggr)ds-H^f(\gamma(d)).
\end{align*}
Using Cauchy-Schwartz inequality we get that
\begin{align}\label{eq4-1}
    0\leq&\int_0^d(m-1)\varphi'(s)^2ds-H^f(\gamma(d)).
\end{align}
Choose $\varphi(s)=\f sd$ and note that $H^f\geq (m-1)K$ on $\pt N$, from \eqref{eq4-1} we have that $d\leq 1/K$. Since the point $x$ is arbitrary, we have proved the inequality \eqref{dist}.

Now we assume that $\pt N$ is compact, then \eqref{dist} implies that $N$ is also compact. By a similar argument in the proof of Lemma \ref{lem-1}, we can prove that $\pt N$ is connected: Suppose not, let $\Sigma$ be one of its components. Choose a $f$-harmonic function $u$ on $N$, which is equal to zero on $\Sigma$ and is equal to one on $\pt N\setminus \Sigma$. Then since $Ric_f^m\geq 0$ on $N$ and $H^f\geq (m-1)K>0$ on $\pt N$, the Reilly formula \eqref{Reilly-2} implies that $\f{\pt u}{\pt \nu}=0$ on $\pt N$, where $\nu$ is the outer unit normal to $\pt N$. By integration by parts, we have that
\begin{align*}
    \int_N|\bar{\nabla}u|^2e^{-f}dv=&-\int_Nu\bar{\Delta}_fue^{-f}dv+\int_{\pt N}u\f{\pt u}{\pt \nu}e^{-f}d\mu=0.
\end{align*}
Therefore $u$ is a constant function on $N$, which is a contradiction since $u=0$ on $\Sigma$ but $u=1$ on $\pt N\setminus\Sigma$.

Suppose the equality  holds in \eqref{dist}, we will show that $N$ is isometric to an $n$-dimensional Euclidean ball. By rescaling the metric of $N$, we may assume $K=1$. Since $M$ is compact, there exists some point $x_0$ in the interior of $N$ such that $d(x_0,\pt N)=1$. It is clear that the geodesic ball $B_1(x_0)$ of radius $1$ centered at $x_0$ is contained in $N$. We claim that $N$ is just the geodesic ball $B_1(x_0)$. In fact, let $\rho=d(x_0,\cdot)$ be the distance function from $x_0$. Since the $m$-Bakry-\'{E}mery Ricci curvature of $N$ is nonnegative, the $f$-Laplacian of $\rho$ satisfies (see equation (4) in \cite{Qian})
\begin{align}\label{f-lap-rho}
    \bar{\Delta}_f(\rho)\leq&\f{m-1}{\rho},
\end{align}
in the sense of distribution. Let $\Sigma=\{q\in\pt N:\rho(q)=1\}$, which is clearly a closed set in $\pt N$ by the continuity of $\rho$. Since $\pt N$ is connected, to show $\Sigma=\pt N$, it suffices to show that $\Sigma$ is also open in $\pt N$, that is for any $q\in \Sigma$, there is an open neighborhood $U$ of $q$ in $\pt N$ such that $\rho\equiv 1$ on $U$. If $q$ is not a conjugate point to $x_0$ in $N$, then the geodesic sphere $\pt B_1(x_0)$ is a smooth hypersurface near $q$ in $N$. Since $\bar{\Delta}\rho$ and $\bar{\nabla}\rho$ are the mean curvature and outer unit normal of the geodesic sphere, we have $\bar{\Delta}_f(\rho)=H^f$. Note that $\rho=1$ on the geodesic sphere and by the $f$-Laplacian comparison inequality \eqref{f-lap-rho}, the $f$-mean curvature of the geodesic sphere is at most $m-1$. However, by the assumption of Theorem \ref{thm-4}, the $f$-mean curvature of $\pt N$ is at least $m-1$. Then from the maximum principle (see \cite{Esch}), we have that $\pt N$ and $\pt B_1(x_0)$ coincides in a neighborhood of $q$.  This implies that $\Sigma$ is open near any $q$ which is not a conjugate point. A similar process in \cite{Martin} (see also Calabi \cite{Cal}) makes us to work through the argument to conclude that $\rho$ is constant near $q$ in $\pt N$, when $q$ is a conjugate point of $x_0$. This proves that $\pt N$ is just the geodesic sphere $\pt B_1(x_0)$ and $M$ is the geodesic ball $B_1(x_0)$.

Next we show that $N$ is isometric to the Euclidean ball of radius one. Since any $q\in \pt N$ can be joined by a minimizing geodesic $\gamma$ parameterized by arc-length from $x_0$ to $q$, and $\gamma$ is orthogonal to $\pt N=\pt B_1(x_0)$ at $q$, which implies that $\gamma$ is uniquely determined by $q$ and $q$ is not in the cut locus of $x_0$. So that $\rho=d(x_0,\cdot)$ is smooth up to the boundary $\pt N$.Then the $f$-Laplacian comparison inequality \eqref{f-lap-rho} holds in the classical sense. Let $|N|=\int_Ne^{-f}dv$ and $|\pt N|=\int_{\pt N}e^{-f}d\mu$ be volumes of $N$ and $\pt N$ with respect to the weighted measure. From the facts that $|\bar{\nabla}\rho|=1$ on $N$, $\rho= 1$ and $\f{\pt\rho}{\pt\nu}=1$ on $\pt N$, the integration by part implies that
\begin{align*}
    |\pt N|-|N|=&\int_{\pt N}\rho\f{\pt\rho}{\pt\nu}e^{-f}d\mu-\int_N|\bar{\nabla}\rho|^2e^{-f}dv\\
    =&\int_{N}\rho\bar{\Delta}_f(\rho)e^{-f}dv~\leq~ (m-1)|N|.
\end{align*}
This implies $|\pt N|\leq m|N|$.

On the other hand, by a similar argument in the proof of Theorem 1 in \cite{Ros}, we can prove that $|\pt N|\geq m|N|$: Let $u$ be a smooth solution of the following Dirichlet problem
\begin{align*}
    \left\{\begin{array}{ll}
             \bar{\Delta}_fu=1, & \textrm{in } N, \\
             u=0, & \textrm{on } \pt N.
           \end{array}\right.
\end{align*}
Integration by part gives that
\begin{align}\label{eq4-2}
    |N|=\int_N\bar{\Delta}_fue^{-f}dv=\int_{\pt N}\f{\pt u}{\pt\nu}e^{-f}d\mu.
\end{align}
Since $Ric_f^m\geq 0$ in $N$ and $H^f\geq (m-1)$ on $\pt N$ (note that we have assumed $K=1$), substituting $u$ to Reilly formula \eqref{Reilly-2} gives that
\begin{align}\label{eq4-3}
    |N|\geq m\int_{\pt N}(\f{\pt u}{\pt \nu})^2e^{-f}d\mu.
\end{align}
From \eqref{eq4-2}, H\"{o}lder inequality and \eqref{eq4-3}, it follows that
\begin{align*}
    |N|^2=&\left(\int_{\pt N}\f{\pt u}{\pt\nu}e^{-f}d\mu\right)^2\\
    \leq&|\pt N|\int_{\pt N}(\f{\pt u}{\pt\nu})^2e^{-f}d\mu\\
    \leq&|\pt N||N|/m
\end{align*}
and we have $|\pt N|\geq m|N|$.

Therefore we get the equality $|\pt N|=m|N|$. Then the equality holds in \eqref{eq4-3} and therefore the Reilly formula \eqref{Reilly-2} assumes equality too. By a similar argument as in the last part in section 4, we get $m=n$, $f$ is constant and $Ric_f^m=Ric$. Then Theorem 1 in \cite{Ros} implies that $N$ is isometric to an Euclidean ball. This completes the proof of Theorem \ref{thm-4}.


\bibliographystyle{Plain}

\end{document}